\documentclass[a4paper, 12pt]{article}

\usepackage[sort&compress]{natbib}
\bibpunct{(}{)}{;}{a}{}{,} 

\usepackage{amsthm, amsmath, amssymb, mathrsfs, multirow, url, subfigure}
\usepackage{graphicx} 
\usepackage{ifthen} 
\usepackage{amsfonts}
\usepackage[usenames]{color}
\usepackage{fullpage}
\usepackage{tikz}

\theoremstyle{plain}

\newtheorem{prop}{Proposition}

\theoremstyle{definition}
\newtheorem{defn}{Definition}

\theoremstyle{remark}

\newcommand{\prob}{\mathsf{P}} 
\newcommand{\E}{\mathsf{E}}

\newcommand{\nm}{{\sf N}}

\newcommand{\chisq}{{\sf ChiSq}}

\newcommand{\RR}{\mathbb{R}}
 
\newcommand{\U}{\mathcal{U}}

\newcommand{\FF}{\mathbb{F}}

\newcommand{\YY}{\mathbb{Y}}

\newcommand{\TT}{\mathbb{T}}
\newcommand{\T}{\mathcal{T}}

\newcommand{\eps}{\varepsilon}

\newcommand{\prior}{\mathsf{Q}}

\newcommand{\cred}{\mathscr{C}}

\newcommand{\lPi}{\underline{\Pi}}
\newcommand{\uPi}{\overline{\Pi}}

\newcommand{\lGamma}{\underline{\Gamma}}

\newcommand{\lprob}{\underline{\mathsf{P}}}
\newcommand{\uprob}{\overline{\mathsf{P}}}

\newcommand{\ytdomain}{\mathfrak{F}}

\title{Fisher's underworld and the behavioral--statistical reliability balance in scientific inference}
\author{Ryan Martin\footnote{Department of Statistics, North Carolina State University, {\tt rgmarti3@ncsu.edu}}}
\date{\today}

\begin{document}

\maketitle 

\begin{abstract}
That science and other domains are now largely data-driven means virtually unlimited opportunities for statisticians.  With great power comes responsibility, so it's imperative that statisticians ensure that the methods being developing to solve these problems are reliable.  But reliable in what sense?  This question is problematic because different notions of reliability correspond to distinct statistical schools of thought, each with their own philosophy and methodology, often giving different answers in applications.  To achieve the goal of reliably solving modern problems, I argue that a balance in the behavioral--statistical priorities is needed.  Towards this, I make use of Fisher's ``underworld of probability'' to motivate a new property called {\em invulnerability} that, roughly, requires the statistician to avoid the risk of losing money in a long-run sense.  Then I go on to make connections between invulnerability and the more familiar behaviorally- and  statistically-motivated notions, namely coherence and (frequentist-style) validity.  

\smallskip

\end{abstract}

\section{Introduction}
\label{S:intro}

Science is now largely data-driven.  This means that statisticians---or people doing statistical analyses---play an important role in the advancement of scientific knowledge.  But substantial care is needed in every aspect of a scientific investigation, statistical and non-statistical, so simply ``doing statistical analyses'' isn't enough.  It's imperative that the process by which data and other available information are fused in order to quantify uncertainty about the scientifically relevant unknowns be {\em reliable}.  The statement I just made is uncontroversial, but only because I gave no specifics that one could disagree with.  Attempts to be more specific face immediate setbacks, for example:
\begin{quote}
{\em ...as to what probability is and how it is connected with statistics, there has seldom been such complete disagreement and breakdown of communication since the Tower of Babel.} \citep[][p.~2]{savage1972} 
\end{quote}
\begin{quote}
{\em Should probability enter to capture degrees of belief about claims? ... Or to ensure we won't reach mistaken interpretations of data too often in the long run of experience?} \citep[][p.~xi]{mayo.book.2018} 
\end{quote} 
These different roles mentioned by Mayo that probability plays correspond to distinct statistical schools of thought, so this hits at the very core of the subject.
\begin{quote}
{\em Two contending philosophical parties, the Bayesians and the frequentists, have been vying for supremacy over the past two-and-a-half centuries...
Unlike most philosophical arguments, this one has important practical consequences. The two philosophies represent competing visions of how science progresses and how mathematical thinking assists in that progress.} \citep{efron2013.ams}
\end{quote}
The longstanding stalemate between the two schools of thought, their distinct views on probability, and their competing notions of reliability suggests that neither is fully satisfactory.  A resolution, therefore, can only come through compromise, through balancing the distinct priorities.  This balancing act requires new perspectives on probability, and the present paper's contribution is along these lines.  




The ``underworld'' in the present paper's title refers to one of Sir Ronald Fisher's papers, entitled ``The underworld of probability'' \citep{fisher.underworld}.  There, Fisher argues that probability has levels or ranks: that there can be probability statements, probability statements about probability statements, etc., and, furthermore, that one agent might be able to make probabilistic wagers, not on the outcome of the experiment directly, but on whether or not another agent's probabilistic wagers on that outcome pay off.  For illustration, Fisher considers the familiar situation involving the toss of a six-sided die, with one of the faces identified as {\em Ace}.  What follows is roughly how \citet{fisher.underworld} described the first few levels of the underworld:
\begin{itemize}
\item {\em Level 1 uncertainty}. Agent~1 is sure that the probability of tossing an Ace is $\frac16$ but, of course, is uncertain about whether a particular toss of the die will be an Ace.  He is, however, willing to accept bets against an Ace at, say, odds 4:1.  
\vspace{-2mm}
\item {\em Level 2 uncertainty}. Agent~2 isn't sure what the probability of tossing an Ace is because he doesn't know the specific die that's to be tossed.  He is, however, sure that at least 10\% of the dice in the box from which Agent~1's die will be randomly chosen have probability $\geq \frac15$ of turning up Ace.  So, Agent~2's assessment is that Agent~1's odds against an Ace may be favorable to his opponents; therefore, Agent~2 would accept bets at odds 9:1 that Agent~1---with his ``accept bets against an Ace at odds 4:1'' rule---will have negative capital in the long run.  
\end{itemize}
Note that, in the Level~2 case, the agent may or may not have his own probability assessment of tossing an Ace.  That is, the description above gives insufficient information to pin down a precise estimate of Agent~2's probability of an Ace.  Agent~2 is saying only that (a)~he thinks Agent~1 is wrong and (b)~he can formulate odds at which he'd bet against the success of Agent~1.  Of course, this process need not stop at Agent~2, it can go even deeper into the underworld, getting less familiar:
\begin{itemize}
\item {\em Level 3 uncertainty}. Agent~3 ``does not know and does not think he knows... the probability of throwing an Ace, but believes he knows how the probability is distributed in a consignment of boxes, one of which has been used to supply the die in question, may know with confidence what odds he can profitably accept in betting against the success of [Agent~2]'s wager'' \citep[][p.~203]{fisher.underworld}.
\end{itemize} 
This process could continue to Level~4, 5, etc., but Level~3 is a deep enough descent into the underworld for our purposes.  Again, Agent~3 need not weigh in on the questions directly relevant to Agents~1--2, his sole objective is capitalizing on a perceived fallibility of Agent~2's assessments, who in turn is hoping to take advantage of a perceived fallibility of Agent~1's assessments.  That is, Agent~3 has side-bets pertaining to the success of Agent~2's side-bets on the success of Agent~1's bets on the die toss.  

How is this practically relevant?  My claim---which I believe was also Fisher's---is that Agents~2--3 in the above scenario have counterparts in the statistical inference problem.  To make sense of this claim, let $\Theta$ denote the probability of tossing an Ace.  Just like in Fisher's scenario above, this value is uncertain to Agent~2, but suppose he is able to quantify his uncertainty by making (possibly imprecise) probability statements about $\Theta$ based on the available information, which might include data, expert opinions, personal beliefs, etc.  This makes Agent~2 like the statistician in a statistical inference problem.  He is, of course, free to quantify his uncertainty about $\Theta$ as he pleases, but if his resulting inferences are to be meaningful in any real-world sense, then they must be capable of standing up to scrutiny from the relevant parties, e.g., society, the scientific community, etc.  In real applications, we only get one shot to draw inferences and, since the truth about $\Theta$ is rarely if ever revealed, this scrutiny can't be based on whether Agent~2 is ``correct'' or not; it must be based on the reliability of his {\em methods}.  Then Agent~3 in Fisher's scenario is the scrutinizer, and if he can identify fallibility or vulnerability in the statistician's uncertainty quantification, i.e., that he can profit off of the perceived systematic errors, then the statistician's probability assessments have been proven unsound or, more precisely, the methods employed to arrive at these probability statements have been proven unreliable.  It's in this sense that Fisher's underworld is practically relevant: it aids in understanding and formulating on a common ground the distinct (probabilistic) objectives of the various parties who have a stake in scientific inference.   


In light of these observations, the statistician concerned with the reliability of his inferences might aim to ensure that no scrutinizer can identify and capitalize off of a lurking vulnerability in his methods.  This puts a constraint on the kinds of methods the statistician can use and what mathematical form his uncertainty quantification takes.  The goal of this paper is to carefully formulate this notion of {\em invulnerability} and to understand the aforementioned constraints and how they can be satisfied.  

Following some background on imprecise probability and statistical inference in Section~\ref{S:background}---in particular, on {\em inferential models} \citep[IMs,][]{imbasics, imbook} for quantifying statistical uncertainty---I'll define a notion of invulnerability for IMs that aligns with the above intuition, i.e., scrutinizers shouldn't be able to capitalize off of fallibility in the statistician's probability assessments.  Then Proposition~\ref{prop:suf} shows that an appropriately defined generalized Bayes IM is invulnerable.  This draws an unexpected connection between Fisher's unique brand of non-subjective, precise-probabilistic reasoning and the largely subjective considerations that dominate the imprecise-probabilistic literature. 


The trouble is that generalized Bayes is rather conservative, so it fails to strike the behavioral--statistical reliability balance.  If invulnerability needs to be relaxed a bit in order to accommodate the statistical reliability considerations, then a relevant question is: in what direction should it be relaxed?  In Section~\ref{S:valid}, I introduce the (partial-prior) validity condition that I've been advocating for recently \citep[e.g.,][]{martin.partial, martin.partial2} and show that validity is a sufficient condition for no sure-loss and a necessary condition for invulnerability.  These results together motivate relaxing invulnerability in the direction of validity.  Moreover, Proposition~\ref{prop:nec} generalizes and also gives a new monetary-loss interpretation to the {\em false confidence theorem} in \citet{balch.martin.ferson.2017}, \citet{martin.nonadditive}, and elsewhere.  That is, the result here shows that if validity fails and the statistician's IM is afflicted with false confidence, then he's at risk of losing money in the long-run.  I hope that this new perspective will help make clear to those who deny the false confidence phenomenon \citep[e.g.,][]{prsa.conf} what's at risk.  

When the available prior information about the relevant unknown is vacuous, validity is equivalent to strong validity and, there, the construction of valid IMs is well-studied.  When the partial prior information is not vacuous, all the IMs I've developed recently \citep[e.g.,][]{martin.partial2} are strongly valid, and it's not yet clear what's the relationship between invulnerability and strong validity.  In Section~\ref{S:balance}, I offer a partial prior IM construction that is valid but not strongly valid; this is based on expressing a (strongly) valid vacuous-prior IM and the partial prior information as belief functions \citep[e.g.,][]{shafer1976, cuzzolin.book} and combining them using Dempster's rule of combination.  Then Proposition~\ref{prop:dempster} establishes that this Dempster's rule-based IM is valid.  A simple numerical example illustrates how this valid partial prior IM helps to strike a behavioral--statistical balance between the vacuous prior IM that ignores the available information and the overly-conservative generalized Bayes IM.


\section{Problem setup}
\label{S:background}

\subsection{Model formulation}
\label{SS:ip}

Let $(Y,\Theta)$ be a pair of uncertain variables, taking values in the possibility space $\YY \times \TT$.  The pair will be modeled by a general lower and upper prevision pair $(\lprob_{Y,\Theta}, \uprob_{Y,\Theta})$ whose domain $\ytdomain$ is the linear space of bounded gambles.  The ``gamble'' terminology comes from the interpretation of $f(Y,\Theta)$ as the uncertain payoff an agent receives when the value of $(Y,\Theta)$ is revealed.  Since the lower and upper previsions are linked via conjugacy, i.e., $\lprob_{Y,\Theta}(f) = -\uprob_{Y,\Theta}(-f)$ for each $f$, it suffices to discuss just one of them, say, the lower prevision $\lprob_{Y,\Theta}$.  When the gamble is an indicator function, $f=1_B$ for $B \subset \YY \times \TT$, I'll write $\lprob_{Y,\Theta}(B)$ instead of $\lprob_{Y,\Theta}(1_B)$.  In this case, conjugacy takes the form 
\[ \lprob_{Y,\Theta}(B) = 1 - \uprob_{Y,\Theta}(B^c), \quad B \subseteq \YY \times \TT. \]

The standard subjective interpretation of an agent's lower/upper prevision pair is as bounds on the prices at which he's willing to buy/sell gambles concerning the uncertain $(Y,\Theta)$.  Indeed, as \citet[][Sec.~2.2]{miranda.cooman.chapter} explain, 
\begin{align*}
\lprob_{Y,\Theta}(f) & = \text{agent's supremum acceptable buying price for the gamble $f$} \\
\uprob_{Y,\Theta}(f) & = \text{agent's infimum acceptable selling price for the gamble $f$}.
\end{align*}
In case the above two prices happen to be the same, that common value would be interpreted as the ``fair price'' for the gamble $f$.  Generalizing this ``fair price'' notion, a gamble $f$ is {\em almost desirable} if $\lprob_{Y,\Theta}(f) \geq 0$. 
For more on lower previsions and desirability, see \citet[][Sec.~3.7.3]{walley1991}, \citet[][Sec.~2.2.3]{miranda.cooman.chapter}, and \citet[][Sec.~1.6.3]{desirability.chapter}.  Restricting gambles to indicators, one could interpret $\lprob_{Y,\Theta}(B)$ and $\uprob_{Y,\Theta}(B)$ as subjective measures of the degree of support and plausibility, respectively, for the assertion or hypothesis ``$(Y,\Theta) \in B$.'' This is consistent with the interpretation in \citet{dempster1967, dempster1968a} and \citet{shafer1976}, even if the lower/upper previsions here don't completely agree with their mathematical formulation.  

To ensure that a lower prevision deserves its ``supremum buying price'' interpretation, certain mathematical structure is needed.  Following \citet[][Sec.~2.5.1]{walley1991}, who built on \citet{definetti1937, definetti.book.1972}, I say that the lower prevision $\lprob_{Y,\Theta}$ is {\em coherent} if 
\begin{align}
\sup_{y,\theta} \Bigl[\sum_{k=1}^K \{ f_k(y,\theta) - \lprob_{Y,\Theta}(f_k)\} & - m \{ f_0(y,\theta) - \lprob_{Y,\Theta}(f_0)\} \Bigr] \geq 0 \notag \\
& \text{all integers $m,K \geq 0$ and all $f_0,f_1,\ldots,f_K \in \ytdomain$}. \label{eq:coherent} 
\end{align}
The case where $m=0$ corresponds to what's called {\em no sure-loss}, which implies that any finite sum of almost desirable gambles is almost desirable.  The general coherence condition \eqref{eq:coherent} simply means that the agent's ``supremum'' buying price $\lprob_{Y,\Theta}(f_0)$ for $f_0$ can't be raised by simultaneously considering a positive linear combination of other almost desirable gambles \citep[e.g.,][Sec.~2.2.1]{miranda.cooman.chapter}. Since virtually all the standard imprecise probability models are coherent, I'll assume throughout this paper, with (virtually) no loss of generality, that $\lprob_{Y,\Theta}$ is a coherent lower prevision.  

There are a number of key consequences of coherence, and here I'll mention just a few.  First, Proposition~4.7 in \citet{lower.previsions.book} says that coherence implies $\lprob_{Y,\Theta}(f) \leq \uprob_{Y,\Theta}(f)$ for any gamble $f \in \ytdomain$, which justifies the ``lower'' and ``upper'' adjectives.  Second, if the lower prevision is coherent and agrees with its corresponding upper prevision, i.e., if $\lprob_{Y,\Theta} = \uprob_{Y,\Theta}$, then the common function, say $\prob_{Y,\Theta}$ is called a linear prevision, or precise probability model, and $\prob_{Y,\Theta}$ reduces to an expectation operator with respect to a finitely additive probability \citep[e.g.,][Sec.~2.8.1]{walley1991}.  Third, $\lprob_{Y,\Theta}$ equals the lower envelope of a closed and convex (credal) set of linear previsions \citep[e.g.,][Prop.~3.3.3]{walley1991}. For more details, see \citet{miranda.cooman.chapter}.  

A common example in the statistical literature is to model $Y$, given $\Theta=\theta$, by a precise probability/linear prevision $\prob_{Y|\theta}$ and a vacuous prior for $\Theta$.  While the ``no prior'' case can't be accommodated within precise probability theory, it can easily be handled using imprecise probability.  Indeed, a minimal extension \citep[][Sec.~6.7]{walley1991} of these individual lower previsions to a coherent lower prevision for $(Y,\Theta)$ is given by 
\[ \lprob_{Y,\Theta}(f) = \inf_{\theta \in \TT} \prob_{Y|\theta}\{ f(Y,\theta) \}, \quad f \in \ytdomain, \]
i.e., the infimum of the $\prob_{Y|\theta}$-expectation of $f(Y,\theta)$ over all $\theta \in \TT$.  For example, if $f=1_B$ is an indicator gamble and $B = A \times H$ is a rectangle in $\YY \times \TT$, then 
\[ \lprob_{Y,\Theta}(B) = \inf_{\theta \in H} \prob_{Y|\theta}(A) \quad \text{and} \quad \uprob_{Y,\Theta}(B) = \sup_{\theta \in H} \prob_{Y|\theta}(A). \]

Despite being common in the statistics literature, the vacuous or ``no prior'' case above is rather extreme.  The general imprecise probability formulation is able to handle all the intermediate/less-extreme cases where, for example, something is known about $\Theta$ that's relevant to the data analysis, but not enough to justify a precise prior distribution for $\Theta$ and the corresponding Bayesian analysis.  This, of course, is related to robustness \citep[e.g.,][]{huber1981}, especially Bayesian prior robustness \citep[e.g.,][]{berger1984}, but there are other motivations for imprecision in statistical applications, as explained in, e.g., \citet[][Ch.~7]{walley1991} and \citet[][Ch.~7]{imprecise.prob.book}; see, also, \citet{martin.partial, martin.partial2}.

\subsection{Statistical inference}
\label{SS:stat}

In statistical applications, the pair $(Y,\Theta)$ decomposes as an observable data point $Y$ and an (unobservable) unknown to be inferred.  When the value $y$ of $Y$ is observed, the goal is to quantify uncertainty about $\Theta$, given $Y=y$.  
\begin{quote}
{\em Statisticians want numerical measures of the degree to which data support hypotheses.} \citep[][p.~122]{hacking.logic.book}
\end{quote}
Towards this, the statistician will choose a mapping $y \mapsto (\lPi_y, \uPi_y)$ that takes a given value of $y$ and returns a data-dependent lower/upper prevision for $\Theta$; the domain $\ytdomain_y$ is the linear space consisting of bounded gambles $f_y(\theta) = f(y,\theta)$ on $\TT$.  I refer to this mapping as an {\em inferential model}, or IM, and I assume that the collection $\{\lPi_y: y \in \YY\}$ are separately coherent lower previsions \citep[e.g.,][Def.~6.2.2]{walley1991}.  That is, for each $y$, $\lPi_y$ is a coherent lower prevision in the sense that
\begin{align*}
\sup_{\theta} \Bigl[\sum_{k=1}^K \{ f_{k,y}(\theta) & - \lPi_y(f_{k,y})\} - m \{ f_{0,y}(\theta) - \lPi_y(f_{0,y})\} \Bigr] \geq 0 \\
& \text{all integers $m,K \geq 0$ and all $f_{0,y},f_{1,y},\ldots,f_{K,y} \in \ytdomain_y$}, 
\end{align*}
and, moreover, $\lPi_y(\{y\} \times \TT) = 1$. Of course, the IM's output can (and will) depend on aspects of the posited model $(\lprob_{Y,\Theta}, \uprob_{Y,\Theta})$, but I'll suppress that dependence in the notation.  The lower/upper prevision will be used as follows: for any relevant hypothesis ``$\Theta \in H$'' about the unknown, the statistician evaluates $\lPi_y(H)$ and $\uPi_y(H)$ and, based on these magnitudes, makes their judgments/inferences.  

Interpretation of the IM is the same as the lower/upper previsions described above.  That is, the IM's numerical output represents lower and upper bounds on the prices deemed acceptable for relevant gambles, i.e., given $Y=y$, 
\begin{align*}
\lPi_y(f_y) & = \text{statistician's supremum buying price for the gamble $f_y$} \\
\uPi_y(f_y) & = \text{statistician's infimum selling price for the gamble $f_y$}.
\end{align*}
In other words, given $Y=y$, the statistician is willing to pay at most $\lPi_y(f_y)$ dollars to a speculator in exchange for $f_y(\Theta)$ dollars; similarly, he's willing to accept at least $\uPi_y(f_y)$ dollars from the speculator in exchange for $f_y(\Theta)$ dollars.  In the familiar case of a precise probability, the lower and upper previsions are the same and their common value, $\Pi_y(f_y)$, say, would be the statistician's fair price for $f_y$.  When restricted to indicator gambles, $\lPi_y(H)$ and $\uPi_y(H)$ could represent, respectively, the statistician's degree of belief and of plausibility in the truthfulness of the claim ``$\Theta \in H$,'' given $Y=y$.  

From the aforementioned broad form of (imprecise-)probabilistic uncertainty quantification, it's possible to derive statistical procedures for hypothesis testing and point/set estimation, if desired.  For example, to test the hypothesis ``$\Theta \in H$,'' corresponding to a fixed subset $H \subseteq \TT$, it would be natural to ``reject $H$'' based on data $Y=y$ if and only if $\uPi_y(H)$ is smaller than some specified threshold.  It's important to have access to these procedures, but it won't be relevant to the present discussion, so I won't get into these details; the interested reader can consult, e.g., \citet{walley2002}, \citet{imbasics}, \citet{augustin.etal.bookchapter}, and/or \citet{imchar, martin.partial, martin.partial2}.  

There are an unlimited number of IMs for the statistician to choose from.  Common choices of IM construction in the statistical literature, mostly in the case of vacuous prior information about $\Theta$, include Fisher's fiducial argument \citep[e.g.,][]{zabell1992, seidenfeld1992, fisher1935a}, default-prior Bayes \citep[e.g.,][]{jeffreys1946}, and generalized fiducial \citep[e.g.,][]{hannig.review}.  Still in the case of vacuous prior information, there are the random set-driven IM constructions in \citet{imbasics, imcond, immarg, imbook} and the more direct approaches in \citet{plausfn, gim}; see, also, \citet{cella.martin.imrisk}.  A generalization allowing non-vacuous prior information, is presented in \citet{martin.partial2}.  

Arguably the most natural IM construction is based on updating the posited model for $(Y,\Theta)$ to account for the observed $Y=y$, analogous to Bayes's rule in ordinary/precise probability.  As expected, this is more complicated than the traditional presentations of Bayes's theorem, and Chapters~6--7 of \citet{walley1991} are devoted to conditioning in the context of general lower/upper previsions.  Walley's focus is largely on what kind of properties are required of the posited (joint/marginal) lower prevision $\lprob_{Y,\Theta}$ to ensure the existence (and uniqueness) of a collection $\{\lprob_{\Theta|y}: y \in \YY\}$ of conditional lower previsions that are separately coherent and also jointly coherent with $\lprob_{Y,\Theta}$.  Here's a summary.  

Consider the partition of $\YY \times \TT$ determined by the $y$-slices 
\begin{equation}
\label{eq:partition}
\mathscr{B} = \bigl\{ B_y=\{y\} \times \TT: y \in \YY \bigr\}. 
\end{equation}
Throughout, when discussing the generalized Bayes IM, I'll (sometimes silently) assume that the posited lower prevision $\lprob_{Y,\Theta}$ is $\mathscr{B}$-{\em conglomerable}, i.e., 
\begin{quote}
for any countable collection $\{B_{y_k}: k=1,2,\ldots\}$ of (distinct) subsets in the partition $\mathscr{B}$, if $\lprob_{Y,\Theta}(1_{B_{y_k}} f) > 0$ for all $k$, then $\lprob_{Y,\Theta}(1_{\bigcup_k B_{y_k}} f) \geq 0$. 
\end{quote}
In words, this says that, if the restriction of $f$ to each set in the partition is desirable, then $f$ itself must also be almost desirable.  Conglomerability has a number of far-reaching consequences and \citet[][Sec.~6.8.4]{walley1991} argues that full conglomerability (i.e., $\mathscr{B}$-conglomerability for every partition $\mathscr{B}$) ought to be a requirement for a rational model $\lprob_{Y,\Theta}$.  Moreover, conglomerability is a relatively weak requirement: it's automatic when $\mathscr{B}$ is finite, which is guaranteed when $\YY$ is finite, and also when $\lprob_{Y,\Theta}(B)=0$ for all $B \in \mathscr{B}$, which is often the case if $\YY$ is uncountable.  For more on conglomerability and its consequences, see \citet{dubins1975}, \citet{ssk.extent.of.noncong}, \citet{kss.finite.additive}, \citet{ssk.noncong.fap}, \citet{miranda.zaffalon.2013, miranda.zaffalon.full}, and \citet{berti.etal.conglomerability}.  

The key point is that conglomerability is a necessary and sufficient condition for existence of separately coherent conditional lower previsions $\{\lprob_{\Theta|y}: y \in \YY\}$ that are also jointly coherent with $\lprob_{Y,\Theta}$; see Theorem~6.8.2(a) in \citet{walley1991}, Section~5.3 in \citet{berti.etal.conglomerability}, and elsewhere.  Beyond existence, Walley's Theorem~6.8.2(c) describes the minimal conditional lower prevision $\{\lGamma_y: y \in \YY\}$ that's coherent with $\lprob_{Y,\Theta}$, which I refer to here as the {\em generalized Bayes IM}, defined as follows.  For each $y \in \YY$: 
\begin{itemize}
\item if $\lprob_{Y,\Theta}(\{y\} \times \TT) = 0$, then 
\[ \lGamma_y(f_y) = \inf_{\theta \in \TT} f(y,\theta), \quad f \in \ytdomain; \] 
\vspace{-8mm}
\item if $\lprob_{Y,\Theta}(\{y\} \times \TT) > 0$, then $\lGamma_y(f_y)$ is given by the {\em generalized Bayes rule} \citep[e.g.,][Theorem~6.4.2]{walley1991}, i.e., 
\[ \lGamma_y(f_y) = \min\Bigl\{ \frac{\prob_{Y,\Theta}(f \, 1_{\{y\} \times \TT})}{\prob_{Y,\Theta}(1_{\{y\} \times \TT})}: \prob_{Y,\Theta} \in \cred(\lprob_{Y,\Theta}) \Bigr\}, \quad f \in \ytdomain, \]
where $\cred(\lprob_{Y,\Theta})$ is the (credal) set of all linear previsions $\prob_{Y,\Theta}$ that dominate the lower prevision $\lprob_{Y,\Theta}$.  The ratio in the above display is the familiar expectation of the gamble $f_y(\Theta)$ with respect to the Bayesian posterior for $\Theta$, given $Y=y$. 
\end{itemize}

\section{Invulnerability}
\label{S:underworld}

Let's return now to Fisher's underworld, but specifically in the context of statistical inference, where Agent~2 is the statistician and Agent~3 is the scrutinizer.  The scrutinizer knows both the data $y$ and the statistician's IM, so in principle he could probe the statistician's assessments and determine if there's any vulnerability, if there are any weaknesses he can potentially profit off of in side-bets.  It's not all about the money for the strutinizer; he's also playing an important role in the scientific process: 
\begin{quote}
{\em Still greater onus is laid on men at large to criticize and probe the evidence given by [statisticians], to question whether they were men trained to observe...} \citep[][p.~54]{pearson.grammar}
\end{quote}
There are a couple relevant questions concerning the meaning of ``profit off of'' that the reader may be asking his/herself, so let me address those here first before moving on.
\begin{itemize}
\item {\em How is Fisher's ``profit in the long-run'' notion relevant to scientific inference?} 

Arguments for the single-$y$ perspective, opposed to a long-run-over-multiple-$y$ perspective, are initially compelling.  That is, why should other values of $Y$ that weren't observed have any bearing on the inferences drawn?  But the issue here is that, as Popper argued, individual probability statements can't be falsified, so the statistician---even with his IM that's coherent for the given $y$---can't be wrong.  Moreover, $\Theta$ will never be revealed, so even the inferences drawn by the statistician can never be proven wrong.  This creates a false sense of security, reinforces attitudes like in the familiar meme ``Being a statistician means never having to say you're certain,'' and breeds carelessness.  The ``long-run'' perspective, however, puts the burden on the statistician to choose an IM whose uncertainty quantification is provably reliable, i.e., his inferences tend to be right.  Finally, the aforementioned reliability is just a probability evaluation with respect to the statistician's posited model which, by the way, doesn't require a ``long-run'' interpretation.  


\item {\em Why does the statistician care if a scrutinizer can play side-bets and profit off of his perceived vulnerabilities?}

If it's not the statistician's money going into the scrutinizer's pocket, then why should he care?  I'll show below that the statistician's interpretation of his IM output (cf.~Section~\ref{SS:stat}) obligates him to participate in these side-bets with the scrutinizer.  So, if there are vulnerabilities in the statistician's uncertainty quantification, then {\em it is his money} going in the scrutinizer's pocket!  It's precisely this connection to the scrutinizer that forces the statistician to have skin in the game \citep[e.g.,][]{taleb.skin, crane.fpp}, i.e., something he personally stands to lose if he's wrong.  This skin in the game is broadly beneficial to society---it forces the statistician to be careful.  
\end{itemize} 

To formalize this, recall that, for any $H \subseteq \TT$ and $\beta \in [0,1]$, if data $Y$ is such that $\lPi_Y(H) > \beta$, then the statistician would be willing to pay the scrutinizer \$$\beta$ for \$$1(\Theta \in H)$.  So, if the scrutinizer believes that the statistician's IM tends to assign too large of $\lPi_Y(H)$ values, and can roughly gauge how large is ``too large,'' then he has a chance to profit off of the statistician's miscalculation.  Towards this, define the function
\begin{equation}
\label{eq:W}
f^{H,\beta}(y,\theta) = \{1(y \in \YY, \theta \in H) - \beta\} \, 1\{\lPi_y(H) > \beta, \theta \in \TT\}.
\end{equation}
Then $f^{H,\beta}(Y,\Theta)$ represents the statistician's payoff on this $(H,\beta)$-dependent gamble, which is an uncertain variable both as a function of $(Y,\Theta)$ and as a function of $\Theta$ with $Y=y$ fixed.  It's not difficult to verify that, 
\[ \lPi_y(f_y^{H,\beta}) \geq 0, \quad \text{for all $H \subseteq \TT$, all $\beta \in [0,1]$, and all $y \in \YY$}, \]
where
\[ f_y^{H,\beta}(\theta) = f^{H,\beta}(y,\theta), \quad \theta \in \TT. \]
That is, for any $(H,\beta)$ pair, the gamble $f_y^{H,\beta}$ is almost desirable to the statistician, relative to his IM's fixed-$y$ assessments, for each $y$.  It's in this sense that the statistician is inclined to participate in the scrutinizer's side-bets, exposing him to potential loss in the event that his IM is vulnerable.  The following definition spells this out explicitly.  

\begin{defn}
\label{def:iv}
Given a coherent model $(\lprob_{Y,\Theta}, \uprob_{Y,\Theta})$ for $(Y,\Theta)$, the statistician's IM is {\em invulnerable} if the gamble $f^{H,\beta}$ in \eqref{eq:W}, which depends implicitly on the aforementioned IM, is almost desirable with respect to $\lprob_{Y,\Theta}$ for all $H$ and all $\beta$, i.e., if 
\[ \lprob_{Y,\Theta}(f^{H,\beta}) \geq 0 \quad \text{for all $H \subseteq \TT$ and $\beta \in [0,1]$}; \]
otherwise, the IM is {\em vulnerable}. 
\end{defn}

To the most cautious of statisticians, namely, those who seek protection against all scrutinizers, invulerability would be a very natural condition.  Basically, invulnerability says that a scrutinizer can offer no gamble of the form \eqref{eq:W} that's acceptable to the statistician as a function of $\Theta$ for fixed $y$ but unacceptable on average as a function of $(Y,\Theta)$.  There are some connections to coherence, which I'll discuss below. 

For some further intuition behind the invulnerability condition, I'd like to make an analogy to Kerckhoff's principle in cryptography:
\begin{quote}
{\em [D]esign [crypto]systems under the assumption that the enemy will immediately gain full familiarity with them.} \citep{shannon.secrecy.1949}
\end{quote}
That is, Kerckhoff's principle says assume that attackers are familiar with your encryption system and design it to be secure regardless of attackers' familiarity.  This is the guiding principle behind commonly used public key encryption systems.  Like Kerckhoff, invulnerability implies protection from loss even when the scrutinizer has full knowledge of the statistician's IM and the time/skills to probe for potential vulnerabilities.  

Despite being a natural or intuitive requirement, there are relevant questions concerning invulnerability that need to be answered.  The first basic question is if an invulnerable IM even exists, and the following proposition answers this in the affirmative.  Indeed, under mild conditions, the generalized Bayes IM in Section~\ref{SS:stat} is invulnerable. 

\begin{prop}[generalized Bayes is invulnerable]
\label{prop:suf}
Suppose that the model $\lprob_{Y,\Theta}$ for $(Y,\Theta)$ is $\mathscr{B}$-conglomerable, where $\mathscr{B}$ is the $y$-slice partition in \eqref{eq:partition}.  Let $\{\lGamma_y: y \in \YY\}$ denote the generalized Bayes IM, whose existence is implied by the assumed conglomerability.  If the statistician's IM $y \mapsto (\lPi_y, \uPi_y)$ satisfies $\lPi_y \leq \lGamma_y$ for each $y$, then it's invulnerable in the sense of Definition~\ref{def:iv}. 
\end{prop}

\begin{proof}
For any fixed $H \subseteq \TT$ and $\beta \in [0,1]$, define the function 
\[ \phi(y,\theta) = 1\{ \lPi_y(H) > \beta, \theta \in \TT\}. \]
Then, using Proposition~2.2(v) in \citet{miranda.cooman.chapter}, 
\begin{align*}
\lprob_{Y,\Theta}(f^{H,\beta}) & = \lprob_{Y,\Theta}\bigl[ \{1(\Theta \in H) - \beta \} \, \phi(Y,\Theta) \bigr] \\
& = \lprob_{Y,\Theta}\bigl[ \{1(\Theta \in H) - \lGamma_Y(H) + \lGamma_Y(H) - \lPi_Y(H) + \lPi_Y(H) - \beta \} \, \phi(Y,\Theta) \bigr] \\
& \geq \lprob_{Y,\Theta}\bigl[ \{1(\Theta \in H) - \lGamma_Y(H)\} \, \phi(Y,\Theta) \bigr] \\
& \qquad + \lprob_{Y,\Theta}\bigl[ \{\lGamma_Y(H) - \lPi_Y(H)\} \, \phi(Y,\Theta) \bigr] \\
& \qquad + \lprob_{Y,\Theta}\bigl[ \{\lPi_Y(H) - \beta\} \, \phi(Y,\Theta) \bigr]. 
\end{align*}
Now the goal is to show that all three terms in the lower bound above are non-negative.  First, by definition of $\phi$, it follows immediately that 
\[ \lprob_{Y,\Theta}\bigl[ \{\lPi_Y(H) - \beta\} \, \phi(Y,\Theta) \bigr] \geq 0, \]
which takes care of the last/third term.  Second, since $\phi$ is $\mathscr{B}$-measurable, i.e., constant on $y$-slices in $\YY \times \TT$, it follows from Property~8 in Section~6.3.5 in \citet{walley1991} that 
\[ \lprob_{Y,\Theta}\bigl[ \{1(\Theta \in H) - \lGamma_Y(H)\} \, \phi(Y,\Theta) \bigr] \geq 0, \]
which takes care of the first term.  Third, and finally, since $\lGamma_y(H) \geq \lPi_y(H)$ for each $y$ and non-negative gambles are almost desirable, it follows that 
\[ \lprob_{Y,\Theta}\bigl[ \{\lGamma_Y(H) - \lPi_Y(H)\} \, \phi(Y,\Theta) \bigr] \geq 0. \]
Since all three terms in the lower bound are non-negative, it must be that $f_{H,\beta}$ is almost desirable, which proves the claim. 
\end{proof}

Proposition~\ref{prop:suf} provides a sufficient condition for invulnerability; I don't know yet if there are other IMs that are invulnerable too.  Regardless, this establishes an apparently new property satisfied by the generalized Bayes IM, and draws an apparently previously unknown path between Fisher and imprecise probability through the underworld.  

The only downside is that the generalized Bayes IM tends to be quite conservative.  For example, in the extreme case with a vacuous prior for $\Theta$, the generalized Bayes rule returns a vacuous posterior.  It's exactly this conservatism that ensures its coherence and invulnerability.  But even Walley wasn't fully committed to generalized Bayes:
\begin{quote}
{\em One difficulty in using generalized Bayes rule is that the updated previsions it defines may be highly imprecise... This suggests that generalized Bayes rule should be regarded as one of many possible strategies for updating beliefs. It is the only reasonable strategy that uses just the unconditional assessments, but other strategies, which use other assessments, are equally valid and may lead to greater precision.} \citep[][p.~286]{walley1991} 
\end{quote}
In a similar vein, I propose seeking a balance of the behavioral--statistical reliability properties, and this apparently requires a relaxation of the invulnerability constraint.  This begs the two-part question: a relaxation how and in which direction?

\section{Validity}
\label{S:valid}

Taking a step back, recall that the statistician's basic goal is to offer ``reliable'' inference.  Invulnerability is a special type of reliability, achieved by the generalized Bayes IM and perhaps others, but may not be the only meaningful such notion.  What about more basic notions of frequentist-style error rate control?  My previous work along these lines was largely motivated by the following remark:
\begin{quote}
{\em Even if an empirical frequency-based view of probability is not used directly as a basis for inference, it is unacceptable if a procedure\ldots of representing uncertain knowledge would, if used repeatedly, give systematically misleading conclusions}.  \citep{reid.cox.2014}
\end{quote}
That is, since the statistician's IM is a ``procedure of representing uncertain knowledge,'' it's ``unacceptable'' if there's no way to control the error at which wrong inferences are drawn.  The following {\em validity} notion from \citet{martin.partial} makes this idea precise. 

\begin{defn}
\label{def:valid}
Relative to the posited model $(\lprob_{Y,\Theta}, \uprob_{Y,\Theta})$, the statistician's IM $y \mapsto (\lPi_y, \uPi_y)$ for inference on $\Theta$ is {\em valid} if 
\begin{equation}
\label{eq:valid}
\uprob_{Y,\Theta}\{ \lPi_Y(H) > 1-\alpha, \, \Theta \not\in H\} \leq \alpha, \quad \text{for all $H \subseteq \TT$ and all $\alpha \in [0,1]$}, 
\end{equation}
\end{defn}

The interpretation of this result is as follows.  If one is reasoning with data based on the magnitudes of the IM's lower and upper probability output, as is recommended, then there's a risk of making erroneous inference if 
\[ \text{$\lPi_Y(H)$ is large and $\Theta \not\in H$}. \]
There's a similar risk of error if $\uPi_Y(H)$ is small and $\Theta \in H$, but this is handled via the conjugacy in the IM's output.  To avoid Reid and Cox's ``systematically misleading conclusions,'' the validity condition ensures that the probability of the aforementioned risk-creating event be controlled.  That's exactly what \eqref{eq:valid} does.  It's important that ``$\alpha$'' appears both inside and outside the probability statement, so that the threshold used to determine if ``$\lPi_Y(H)$ is large'' is explicitly linked to error rate control achieved.  

The notion of validity in Definition~\ref{def:valid} generalizes that condition with the same name advocated for in, e.g., \citet{imbasics, imbook}.  This previous work focused on the case of a precise model $\prob_{Y|\theta}$ for $(Y \mid \Theta=\theta)$ with a vacuous prior for $\Theta$.  For such a model, the condition in \eqref{eq:valid} specializes to 
\begin{equation}
\label{eq:valid.vac}
\sup_{\theta \not\in H} \prob_{Y|\theta}\{ \lPi_Y(H) > 1-\alpha \} \leq \alpha, \quad \text{for all $H \subseteq \TT$ and all $\alpha \in [0,1]$}, 
\end{equation}
which is the same as that in \citet[][Sec.~3.3]{imbasics}.  What distinguishes the new version of validity from the old is that the latter can accommodate more flexibility in the model, allowing for error rate control with respect to more general priors. 

Although the validity condition has a strong statistical reliability flavor, it's worth asking---in the spirit of striking a balance---if validity has any behavioral reliability implications as well.  Along these lines, the next result says that {\em validity implies no sure-loss}.  More specifically, if one interprets the IM as a rule for updating the prior assessments,
\[ \lprob_\Theta(H) := \lprob_{Y,\Theta}(\YY \times H) \quad \text{and} \quad \uprob_\Theta(H) := \uprob_{Y,\Theta}(\YY \times H), \quad H \subseteq \TT, \]
to a data-dependent ``posterior'' assessment $(\lPi_y, \uPi_y)$ given data $Y=y$, then one would hope that there's no sure-loss, i.e., the IM should satisfy 
\begin{equation}
\label{eq:no.sure.loss}
\inf_{y \in \YY} \lPi_y(H) \leq \uprob_\Theta(H), \quad \text{for all $H \subseteq \TT$}. 
\end{equation}
In words, if \eqref{eq:no.sure.loss} fails, so that the above relation is ``$>$'' for some $H$, then the scrutinizer could buy the gamble \$$1(\Theta \in H)$ from the statistician for \$$\uprob_\Theta(H)$ before $Y$ is observed, and then sell it back to the statistician for \$$\lPi_y(H)$ after $Y=y$ is observed and, regardless of the value $y$, will make a profit of 
\[ \lPi_y(H) - \uprob_\Theta(H) > 0 \quad \text{dollars}. \]
A specific behavioral implication of validity is in Proposition~\ref{prop:no.sure.loss} below, which confirms that if the IM is valid, then sure-loss in the sense described here is avoided. 

\begin{prop}[validity implies no sure-loss]
\label{prop:no.sure.loss}
If the IM is valid in the sense of Definition~\ref{def:valid}, then there's no sure-loss relative to the prior assessments, i.e., \eqref{eq:no.sure.loss} holds.  
\end{prop}

\begin{proof}
Suppose that there's a sure-loss.  That is, there exists an $H \subseteq \TT$ and $\eps > 0$ such that \eqref{eq:no.sure.loss} fails, i.e., $\lPi_y(H) > \uprob_\Theta(H) + \eps$ for all $y$.  Take $\alpha = 1-\uprob_\Theta(H)-\eps$, so that 
\begin{align*}
\uprob_{Y,\Theta}\{\lPi_Y(H) > 1 - \alpha, \, \Theta \not\in H\} & = \uprob_{Y,\Theta}\{\lPi_Y(H) > \uprob_\Theta(H) + \eps, \, \Theta \not\in H\} \\
& = \uprob_\Theta(H^c) \\
& = 1 - \lprob_\Theta(H) \\
& \geq 1 - \uprob_\Theta(H) \\
& = \alpha + \eps \\
& > \alpha,
\end{align*}
where the second equality follows since $\lPi_y(H)$ is lower-bounded by $\uprob_\Theta(H) + \eps$ uniformly in $y$. Therefore, validity \eqref{eq:valid} fails, and the claim follows by contraposition. 
\end{proof}

The next result connects validity to invulnerability.  In contrast with Proposition~\ref{prop:suf}, I conclude here that validity is necessary for invulnerability.  That is, if the statistician's IM isn't valid, then he's exposed to potential loss by a clever/lucky scrutinizer. 

\begin{prop}[invulnerability implies validity]
\label{prop:nec}
Given the model $\lprob_{Y,\Theta}$, if the statistician's IM $y \mapsto (\lPi_y,\uPi_y)$ for $\Theta$ is invulnerable in the sense of Definition~\ref{def:iv}, then it's also valid in the sense of Definition~\ref{def:valid}. 
\end{prop}

\begin{proof}
Suppose that \eqref{eq:valid} fails, and let $(H,\alpha)$ be such that 
\[ \uprob_{Y,\Theta}\{ \lPi_Y(H) > 1-\alpha, \Theta \not\in H\} > \alpha. \]
Set $\beta=1-\alpha$, and consider the gamble $f^{H,\beta}$ as defined above; recall that, by definition, $f_y^{H,\beta}$ is acceptable to the statistician for any given $y$.  Rewrite $f_y^{H\beta}$ as 
\begin{align*}
f_y^{H,\beta}(\theta) & = \{ 1(\theta \in H) - \beta \} \, 1\{ \lPi_y(H) > \beta \} \\
& = \{ \alpha - 1(\theta \not\in H) \} \, 1\{ \lPi_y(H) > 1-\alpha \}.
\end{align*}
Define a new gamble 
\[ \check f_y^{H,\beta}(\theta) = \alpha - 1\{ \lPi_y(H) > 1-\alpha, \theta \not\in H \}, \]
and notice that $f_y^{H,\beta}(\theta) \leq \check f_y^{H,\beta}(\theta)$.  Since the former gamble is acceptable for each $y$, it must be that the latter gamble is too.  But it's also clear that 
\[ \lprob_{Y,\Theta}( \check f^{H,\beta} ) = \alpha - \uprob_{Y,\Theta}\{ \lPi_Y(H) \geq 1-\alpha, \, \Theta \not\in H\}, \]
and the right-hand side is negative for this choice of $(H,\alpha)$.  So, if validity fails, then the IM is vulnerable and, therefore, by contraposition, validity implies no sure-loss.  
\end{proof}

Proposition~\ref{prop:nec} has close connections to the {\em false confidence theorem} in \citet{balch.martin.ferson.2017}.  There we showed that, in the case of vacuous prior information about $\Theta$, if the IM $y \mapsto (\lPi_y, \uPi_y)$ is precise, i.e., if $\lPi_y = \uPi_y$, then it fails to be valid.  Consequently, for any threshold $\alpha$, there exists $H \subseteq \TT$ such that 
\[ \prob_{Y|\theta}\{ \lPi_Y(H) > 1-\alpha \} > \alpha, \quad \text{for some $\theta \not\in H$}. \]
That is, there exists a false hypothesis that the IM tends to assign (relatively) high confidence/support/belief to, hence the name ``false confidence.''  And this has nothing to do with, say, choosing a ``wrong prior'' in a Bayesian analysis---it's a consequence of unjustified precision, of pushing the likelihood too far (as I explain in \citealt{martin.basu}), that applies to any default-prior Bayes or fiducial-like solution. What's new and interesting/important here is the behavioral connection: in the vacuous prior case, a precise IM is invalid, hence vulnerable, and hence a statistician adopting such an IM is at risk of losing money to a scrutinizer.  Having exposed the behavioral/monetary consequences of unjustified imprecision, I hope that those who dismissed the false confidence phenomenon \citep[e.g.,][]{prsa.conf} can now more clearly see the potential consequences.  

This, of course, isn't the first such criticism of unjustifiably precise statistical solutions when the prior information is vacuous.  Chapter~7.4 in \citet{walley1991} has a number of such results; see \citet{seidenfeld1982} and \citet{caprio.seidenfeld.isipta23} for related investigations into fiducial and generalized fiducial inference, respectively; and see, also, \citet{martin.partial, martin.isipta2023}.  What's new and interesting is the connection made here to the seemingly entirely frequentistly-motivated validity property.  

Proposition~\ref{prop:nec} isn't specifically focused on the vacuous prior case, so it also offers a particular generalization of the original false confidence theorem's take-away message.  That is, if the statistician's IM fails to be valid even in the more general sense of Definition~\ref{def:valid} above \citep[from][]{martin.partial}, which allows for partial prior information on $\Theta$, then there's vulnerability and a corresponding risk of monetary loss to a savvy scrutinizer.  This underscores how much care is needed in both the modeling---including prior elicitation---and IM construction steps in a statistical analysis.  

Finally, in light of the above observations, the validity property emerges as perhaps the ``correct'' direction to pursue as it concerns striking a balance between behavioral and statistical reliability.  Indeed, the above results reveal the following relationships:
\begin{align*}
\text{generalized Bayes IM} & \in \{\text{invulnerable IMs}\} \\
& \subseteq \{\text{valid IMs}\} \\
& \subseteq \{\text{no-sure-loss IMs}\}. 
\end{align*}
(By the way, \citet{martin.partial}, Corollary~3, established that the generalized Bayes IM is valid, but Propositions~\ref{prop:suf} and \ref{prop:nec} here together offer a new proof.)  So, if generalized Bayes and invulnerability more generally are taken to be worthy considerations, but the concern is conservatism due to an over-emphasis on behavioral reliability, then I think it makes sense to seek a better balance by relaxing the requirements, by expanding the class of IMs under consideration.  Since one wouldn't want to go so far as to allow IMs that suffer from sure-loss, restricting attention to valid IMs, which have been previously investigated, seems like a very reasonable strategy.

\section{Striking a balance?}
\label{S:balance}

\subsection{Valid IM construction}

Following the suggestion at the end of the previous section, the present goal is to construct an IM that's valid and different from generalized Bayes.  Previous work, e.g., in the \citet{imbook} monograph, focuses on the construction of valid IMs different from generalized Bayes, but there's a subtle point: those previous efforts all focused on the vacuous prior case that's by far the most common in the statistical literature.  Those IMs that are valid relative to a vacuous prior are also valid relative to a more precise partial prior \citep[][Proposition~4]{martin.partial}, but they actually satisfy a stronger form of validity that's not obviously linked to invulnerability via Proposition~\ref{prop:nec}; see Section~\ref{S:discuss} below.  For now, I want to focus on the present goal, namely, the construction of a partial prior IM that's valid in the sense of Definition~\ref{def:valid} above.  

What I'm going to present here is one natural construction of a valid, partial prior IM.  I make no claims that this is the only such construction or the best; further investigation is needed to properly formulate and hopefully answer the question of ``best.'' 

As mentioned above, it's now well-known how to construct an IM that's valid relative to a vacuous prior.  So, can one suitably combine that vacuous-prior IM with the available partial prior information?  This is similar to how Bayesian inference combines the data-driven likelihood function with a prior density, or how frequentists penalize their objective functions to encourage certain structure in their point estimates.  Following the ideas in \citet{martin.partial}, I'll consider using {\em Dempster's rule} to combine the vacuous-prior IM and partial prior information \citep[e.g.,][]{shafer2016.rule, shafer1976, dempster1967, dempster2008}.  

Start with a classical/precise statistical model, $\prob_{Y|\theta}$, for $(Y \mid \Theta=\theta)$; the (partial) prior for $\Theta$ will be discussed below.  Express this precise model $\prob_{Y|\theta}$ via an {\em association} or a ``data-generating equation'' 
\begin{equation}
\label{eq:assoc.new}
Y = a(\theta, U), 
\end{equation}
where $U$ has a known distribution $\prob_U$.  This form is familiar in the context of simulating data from a known distribution.  That is, we can generate data $Y$ from distribution $\prob_{Y|\theta}$, with known $\theta$, by first generating $U$ from distribution $\prob_U$ and then plugging the corresponding $(\theta,U)$ into the expression \eqref{eq:assoc.new} to get $Y$.  Any model $\prob_{Y|\theta}$ can be represented by an expression like \eqref{eq:assoc.new}, but that representation is not unique.  This setup aligns with Fisher's fiducial approach and the generalizations by \citet{fraser1968}, \citet{dempster2008}, and \citet{hannig.review}.  The IM construction in \citet{imbook} differs from these others through its use of a random set $\U \sim \prob_\U$ for quantifying uncertainty about the unobservable auxiliary variable $U \sim \prob_U$ in \eqref{eq:assoc.new}.  Define $\U$ such that the $y$-dependent random set on the parameter space $\TT$, given by
\[ \TT_y(\U) = \bigcup_{u \in \U} \{\theta: y = a(\theta, u)\}, \quad \U \sim \prob_\U, \]
is non-empty for all $y$.  Then define the IM's lower and upper probabilities for quantifying uncertainty about $\Theta$, given $Y=y$, as 
\begin{align*}
\lPi_y^\text{vac}(H) & = \prob_\U\{ \TT_y(\U) \subseteq H \} \\
\uPi_y^\text{vac}(H) & = \prob_\U\{ \TT_y(\U) \cap H \neq \varnothing\}, \quad A \subseteq \TT. 
\end{align*}
The $\lPi_y^\text{vac}$ is a belief function and, if the random set $\U$ is nested, i.e., for any pair of realizations one is a subset of the other, then $\lPi_y^\text{vac}$ is {\em consonant} \citep[e.g.,][Ch.~10]{shafer1976}.  I mention consonance because this structure appears to be fundamental as it concerns the IM's efficiency; see Theorem~4.3 in \citet{imbook} and Section~\ref{S:discuss} below.  Importantly, under very mild conditions on the random set $\U$, it follows that the IM constructed above is valid relative to a vacuous prior for $\Theta$, i.e., \eqref{eq:valid.vac} holds.  For conditions on $\U$ and a proof of the previous claim, see \citet[][Sec.~3]{imbasics}. 

The IM construction just described would be appropriate if prior information about $\Theta$ were vacuous.  If prior information is not vacuous, then it wouldn't be unreasonable to encode this too as a belief function \citep[e.g.,][]{cuzzolin.book, denoeux1999} and combine it with the $y$-dependent belief function above using Dempster's rule.  A convenient way to express this prior belief function is by introducing a (non-empty) random set $\T \sim \prob_\T$ on $\TT$, independent of $U$ and $\U$, so that prior uncertainty can be quantified as 
\[ \lprob_\Theta(H) = \prob_\T\{ \T \subseteq H\} \quad \text{and} \quad \uprob_\Theta(H) = \prob_\T\{ \T \cap H \neq \varnothing \}, \quad H \subseteq \TT. \]
Again, I make no claims that this would be an appropriate model for quantifying uncertainty about $\Theta$ in every case, but it is quite flexible compared to the all-or-nothing vacuous or precise approaches that statisticians are currently taking.  

Now, given data $Y=y$, we can combine the available prior information, encoded in $\T$, with the output of the IM construction above, expressed in terms of the random set $\Theta_y(\U)$, using Dempster's rule to get some new IM output, namely, 
\begin{equation}
\label{eq:dempster}
\lPi_y(H) = 
\frac{\prob_{\T \times \U}\{ \varnothing \neq \TT_y(\U) \cap \T \subseteq H\}}{\prob_{\T \times \U}\{\TT_y(\U) \cap \T \neq \varnothing\}}, \quad H \subseteq \TT, 
\end{equation}
where $\prob_{\T \times \U}$ is the joint distribution for the (independent) random sets $\T$ and $\U$.  The corresponding upper probability/plausibility is 
\begin{equation}
\label{eq:dempster.pl}
\uPi_y(H) = \frac{\prob_{\T \times \U}\{ \TT_y(\U) \cap \T \cap H \neq \varnothing \}}{\prob_{\T \times \U}\{\TT_y(\U) \cap \T \neq \varnothing\}}, \quad H \subseteq \TT.
\end{equation}
Note that this IM is different from both generalized Bayes and Dempster's solutions. The difference from generalized Bayes is probably obvious, but the difference from Dempster's solution stems from introducing the random set $\U$ that's necessary to achieve the vacuous-prior validity in \eqref{eq:valid.vac} which, in turn, appears to be necessary for proving the more general validity claim below.  There's much more that could be said about the proposed IM in \eqref{eq:dempster}, but my focus here is just on proving its validity.  

Towards this, I'm going to make two assumptions concerning the mathematical structure of the available prior information.  
\begin{itemize}
\item[A1.] Assume that the support of $\T$ is finite, i.e., $\T$ has only finitely many focal elements.  This implies that $\prob_\T$ can be described by a mass function $m$:
\[ \uprob_\Theta(H) = \prob_\T( \T \cap H \neq \varnothing ) = \sum_{T: T \cap H \neq \varnothing} m(T), \quad H \subseteq \TT. \]
This is just for mathematical simplicity, the validity result shouldn't hinge on countable versus uncountable support for $\T$. 
\vspace{-2mm}
\item[A2.] Assume that $\T$ is nested---that is, for any $T,T'$ in the support of $\T$, either $T \subseteq T'$ or $T \supseteq T'$.  This implies that the prior belief function $\lprob_\Theta$ is consonant.  Admittedly, this is a non-trivial restriction, but there are many practical cases where this would be appropriate.  Like above, I wouldn't expect that consonant versus non-consonant would be relevant to the validity result, but I don't currently know how to prove the claim without consonance.  
\end{itemize} 

\begin{prop}
\label{prop:dempster}
Let the model $(\lprob_{Y,\Theta}, \uprob_{Y,\Theta})$ be determined by the conditional $\prob_{Y|\theta}$ model for $(Y \mid \Theta=\theta)$ and the partial prior $(\lprob_\Theta, \uprob_\Theta)$ for $\Theta$.  Under Assumptions~A1--A2, the Dempster's rule-based IM construction in \eqref{eq:dempster} is valid in the sense of Definition~\ref{def:valid}. 
\end{prop}

\begin{proof}
Take any fixed focal element $T$ of $\T$.  By monotonicity of $H \mapsto \uPi_y(H)$ for each $y$, it's easy to confirm that 
\begin{equation}
\label{eq:demp.mono}
\prob_{Y|\theta}\{ \uPi_Y(H) \leq \alpha\} \leq \prob_{Y|\theta}\{ \uPi_Y(H \cap T) \leq \alpha\}. 
\end{equation}
Next, note the following lower bound on $\uPi_Y(H \cap T)$ from \eqref{eq:dempster.pl}:
\begin{align}
\uPi_y(H \cap T) & = \frac{\prob_{\T \times \U}\{ \Theta_y(\U) \cap \T \cap (H \cap T) \neq \varnothing\}}{\prob_{\T \times \U}\{\Theta_y(\U) \cap \T \neq \varnothing\}} \notag \\
& \geq \prob_{\T \times \U}\{ \Theta_y(\U) \cap \T \cap (H \cap T) \neq \varnothing\} \notag \\
& = \sum_{T': T' \cap H \cap T \neq \varnothing} m(T') \, \prob_{\T \times \U}\{ \Theta_y(\U) \cap (T' \cap H \cap T) \neq \varnothing\} \notag \\
& = \prob_\U\{ \Theta_y(\U) \cap (H \cap T) \neq \varnothing\} \, \uprob_\Theta(H) \notag \\
& = \uPi_y^\text{vac}(H \cap T) \, \uprob_\Theta(H) \label{eq:demp.lb}
\end{align}
The first inequality follows because the denominator is a probability, hence $\leq 1$; the next equality is by the total probability formula with respect to the distribution of $\T$; the last equality follows by the assumed nestedness of $\T$, i.e., that $T'$ has non-empty intersection with $H \cap T$ if and only if $T'$ has non-empty intersection with $H$.  

Now suppose---without loss of generality---that the focal elements of $\T$ are ordered as $T_k \subset T_{k+1}$, for $k \geq 1$.  Define the expansion 
\[ H = \bigcup_{k \geq 1} \underbrace{H \cap (T_k \setminus T_{k-1})}_{H_k}, \]
where the $H_k$'s are disjoint.  Notice that $H_k \subseteq H \cap T_k$, so if $\theta \in H_k$, then the hypothesis $H \cap T_k$ is ``true'' and it follows from the vacuous-prior validity result \eqref{eq:valid.vac} that  
\begin{equation}
\label{eq:demp.vac}
\sup_{\theta \in H_k} \prob_{Y|\theta}\{ \uPi_Y^\text{vac}(H \cap T_k) \leq \rho \} \leq \min(1, \rho), \quad \text{for any $\rho > 0$}. 
\end{equation}

Finally, with a slight abuse of notation, denote the extension of the prior upper probability/plausibility function for $\Theta$ to an upper prevision (supported on a suitable set of bounded gambles on $\TT$) by the same symbol $\uprob_\Theta$.  This extension is achieved through a simple Choquet integral as shown in, e.g., \citet{denoeux.decision.2019}.  Now, take any linear prevision $\prior$ dominated by the upper prevision $\uprob_\Theta$, i.e., $\prior(\cdot) \leq \uprob_\Theta(\cdot)$.  Then
\begin{align*}
\prior\bigl[ 1(\theta \in H) \, \prob_{Y|\theta}\{ \uPi_Y(H) \leq \alpha\} \bigr]  & = \sum_k \prior\bigl[ 1(\theta \in H_k) \, \prob_{Y|\theta}\{ \uPi_Y(H) \leq \alpha\} \bigr] \\
& \leq \sum_k \prior\bigl[ 1(\theta \in H_k) \, \prob_{Y|\theta}\{ \uPi_Y(H \cap T_k) \leq \alpha\} \bigr] \\
& \leq \sum_k \prior\bigl[ 1(\theta \in H_k) \, \prob_{Y|\theta}\{ \uPi_Y^\text{vac}(H \cap T_k) \leq \alpha \uprob_\Theta(H)^{-1}\} \bigr] \\
& \leq \sum_k \min\{1, \alpha \uprob_\Theta(H)^{-1}\} \, \prior(H_k) \\
& \leq \alpha \, \uprob_\Theta(H)^{-1} \, \prior(H),
\end{align*}
where the equality is by the disjoint-union representation $H = \bigcup_k H_k$ and finite additivity, the first inequality is by monotonicity \eqref{eq:demp.mono}, the second inequality is by \eqref{eq:demp.lb}, the third inequality is by \eqref{eq:demp.vac}, and the fourth inequality by the disjoint-union representation again.  
By definition, $\prior(H) \leq \uprob_\Theta(H)$, so the upper bound is no more than $\alpha$.  Since $\uprob_{Y,\Theta}$ is coherent, it's the upper envelope of the linear previsions $\prob_{Y,\Theta}$ it dominates.  For the present model---precise likelihood, imprecise prior---that upper envelope corresponds to a supremum over all the priors $\prior$ for $\Theta$ dominated by $\uprob_\Theta$.  Therefore, 
\[ \uprob_{Y,\Theta}\{ \uPi_Y(H) \leq \alpha, \Theta \in H \} = \sup_{\prior: \prior \leq \uprob_\Theta} \prior\bigl[ 1(\Theta \in H) \, \prob_{Y|\Theta}\{ \uPi_Y(H) \leq \alpha\} \bigr], \]
and, since the upper bound is no more than $\alpha$, the validity claim follows.
\end{proof}

To recap, I'm proposing a simple IM construction based on treating the vacuous-prior IM and the partial prior as independent belief functions---determined by their own independent random sets---and combining them via Dempster's rule.  This isn't the only IM construction available, but it's a simple and reasonable idea.  Proposition~\ref{prop:dempster} establishes that, besides being ``reasonable,'' this Dempster's rule-based IM construction is valid in the sense of Definition~\ref{def:valid}, at least under certain conditions on the prior specification.

\subsection{Illustration}

To illustrate the potential balance that this alternative to the generalized Bayes IM construction, I'll consider a simple example.  Model $(Y \mid \Theta=\theta)$ as $\prob_{Y|\theta} = \nm(\theta, 1)$, a normal distribution with mean $\theta$ and variance 1.  I'll start here with a quick review of the vacuous prior IM for $\Theta$, given $Y=y$, derived in \citet{imbasics}.  Write the association \eqref{eq:assoc.new} as $Y = \theta + U$, where $U \sim \nm(0,1)$.  Model the unobserved value of $U$ by a belief function determined by the random set 
\[ \U = \{u \in \RR: |u| \leq |U^\star|\}, \quad U^\star \sim \nm(0,1), \]
a symmetric interval around the origin with random length.  Then 
\[ \TT_y(\U) = \bigcup_{u \in \U} \{\theta: y = \theta + u\} = \bigl[ y - |U^\star|, y + |U^\star| \bigr], \quad U^\star \sim \nm(0,1). \]
For a class of hypotheses $(\infty, \theta]$, $\theta \in \RR$, it's easy to verify that 
\begin{align*}
\lPi_y^\text{vac}((-\infty, \theta]) & = \prob( y + |U^\star| \leq \theta ) \\
\uPi_y^\text{vac}((-\infty, \theta]) & = \prob( y - |U^\star| \leq \theta ).
\end{align*}
Individually, these functions have the properties of a cumulative distribution function as $\theta$ varies, for fixed $y$, but of course there's an ordering that holds pointwise in $\theta$. Figure~\ref{fig:cdf} plots these lower and upper distribution functions for a few different values of $y$.

Next, suppose the statistician is ``90\% sure that $\Theta$ is no more than 7.'' This incomplete prior information can be modeled by a belief function determined by the random set $\T$ 
\[ \T = \begin{cases} (-\infty, 7] & \text{with probability 0.9} \\ \RR & \text{with probability 0.1}, \end{cases} \]
so the mass function is 
\[ m((-\infty, 7]) = 0.9 \quad \text{and} \quad m(\RR) = 0.1. \]
It's not too difficult to carry employ Dempster's rule to combine this prior with the vacuous-prior IM described above---the expressions are messy and not worth presenting here---and then find the corresponding lower and upper distributions.  Figure~\ref{fig:cdf} plots these lower and upper distribution functions for a few different values of $y$.  

\begin{figure}[t]
\begin{center}
\subfigure[Clearly compatible: $y=5$]{\scalebox{0.6}{\includegraphics{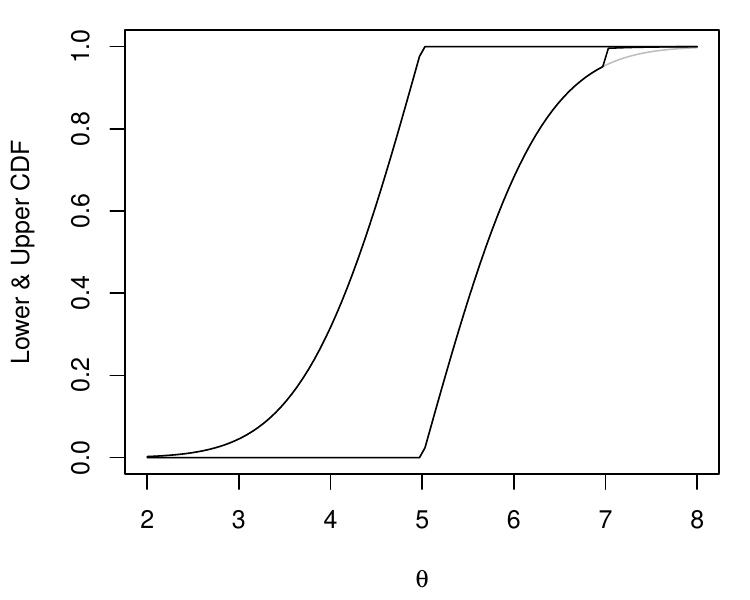}}}
\subfigure[Barely compatible: $y=6.5$]{\scalebox{0.6}{\includegraphics{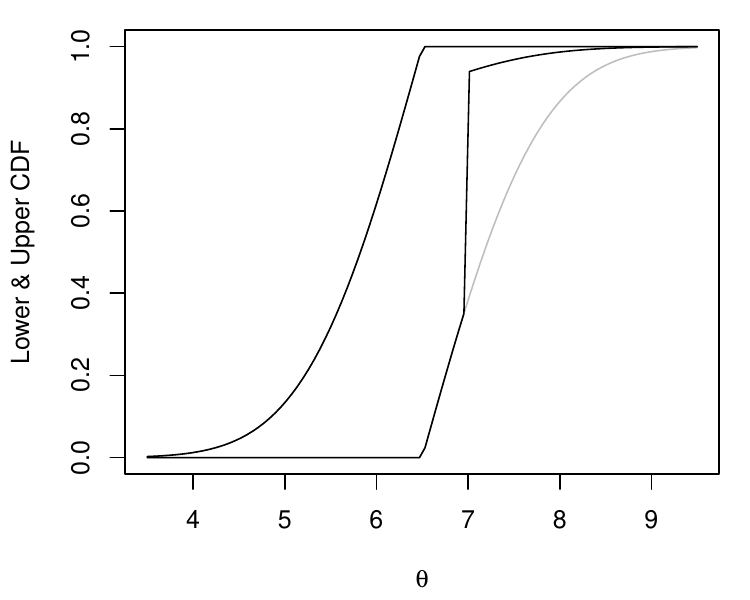}}}
\subfigure[Barely incompatible: $y=7.5$]{\scalebox{0.6}{\includegraphics{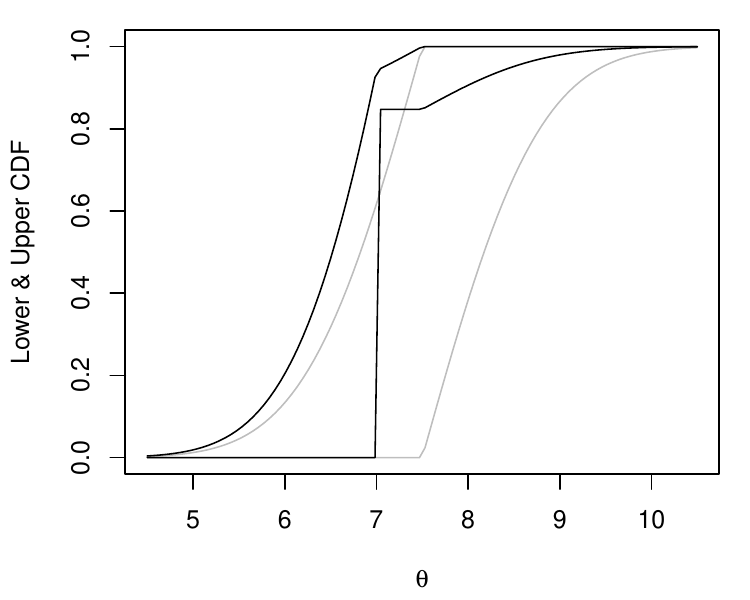}}}
\subfigure[Clearly incompatible: $y=9$]{\scalebox{0.6}{\includegraphics{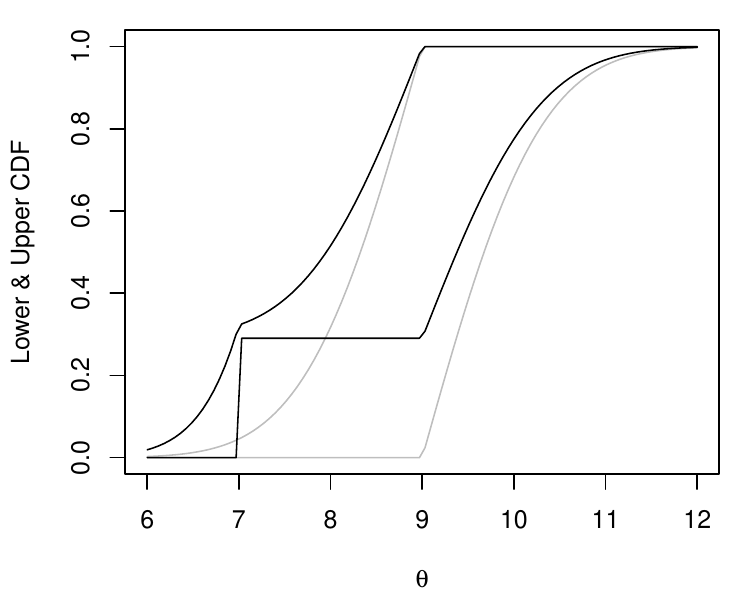}}}
\end{center}
\caption{Plots of the lower and upper distribution functions for the vacuous-prior IM (gray) and Dempster's rule-based IM (black) for four values of $y$ corresponding to four different degrees of compatibility with the partial prior information.}
\label{fig:cdf}
\end{figure}

In Figure~\ref{fig:cdf}, the $y$ values range from being clearly compatible with the prior information, barely compatible, barely incompatible, and clearly incompatible.  In the clearly-compatible case, there's virtually no difference between the two IMs' results, with the Dempster's rule-based solution being slightly more precise.  In the barely-compatible case, the increased precision is more pronounced compared to the previous case because there's more to be gained from the prior information in this boundary case.  When the data is (barely or clearly) incompatible with the prior information, the prior's effect is more significant and, moreover, the difference between the two IMs can no longer be described in terms of relative precision.  To understand the effect in the incompatible cases, it helps to think about a statistical summary, namely, the interval determined by, say, the 2.5th and 97.5th percentiles of the upper and lower distribution functions, respectively.  For the vacuous-prior IM, these intervals have length roughly 4.  In the barely-incompatible case, the incompatibility can easily be chalked up to sampling variation, so the prior steps in and tightens up the aforementioned interval to have length less than 3.5.  In the clearly-incompatible case, on the other hand, the discrepancy can't easily be explained by sampling variability, so, thanks to the validity property (Proposition~\ref{prop:dempster}), the IM responds by stretching out the aforementioned interval to roughly length 5.  

For comparison, the generalized Bayes IM described in Section~\ref{SS:stat} is vacuous in this case.  That is, thanks to the invulnerability and coherence constraints, the generalized Bayes IM is so conservative that it's unable to learn based on the observation $Y=y$.  The Dempster's rule-based IM sacrifices on invulnerability---but not on validity---in exchange for the ability to make non-trivial inferences in common cases like these where detailed prior information is lacking.  This illustrates the balance that I'm hoping to achieve by relaxing the invulnerability requirement in the direction of valid-but-vulnerable IMs.




\section{Conclusion}
\label{S:discuss}

In this paper, I start by considering Fisher's so-called ``underworld of probability,'' which describes---in terms of bets---the different levels or perspectives on uncertainty.  Each of these levels corresponds to an agent having different priorities and different forms of information available.  While Fisher doesn't say so directly, I believe that his motivation for writing about the underworld was to make an analogy between agents in his formulation and (some of) the key players in scientific inference, namely, 
\begin{itemize}
\item the statistician, or the one who quantifies uncertainty for the purpose of drawing inferences about scientifically relevant unknowns based on observed data, and 
\vspace{-2mm}
\item the scrutinizers of scientific inference, e.g., society and/or the scientific community, the ones who stand to lose if the statistician's inferences happen to be erroneous. 
\end{itemize}
Understanding the various priorities is imperative because scientific advances are now largely data-driven.  I use Fisher's underworld considerations as a jumping off point to define a notion of {\em invulnerability} that ensures no scrutinizer can identify and profit off of fallibility in the statistician's uncertainty quantification and inference.  Then Proposition~\ref{prop:suf} establishes that invulnerability can be attained, in particular, by what I call here the generalized Bayes inferential model (IM).  This establishes an unexpected and apparently new connection between Fisher's nuanced/elusive brand of probabilistic reasoning and the behavioral considerations of de~Finetti, Walley, Williams, and many other contributors to the (imprecise) probability literature.  

While having statistical motivations, the close connection between invulnerability and behavioral reliability considerations---and the well-known conservatism of generalized Bayes---suggests a relaxation in some direction might be needed to achieve the desired balance.  Towards this, Proposition~\ref{prop:nec} shows that the validity property (Definition~\ref{def:valid}) that I've been advocating for recently is a necessary condition for invulnerability.  This, together with the dual facts that generalized Bayes IMs are valid and valid IMs avoid sure loss, encourages relaxing the (apparently conservative) notion of invulernability in the direction of validity.  Section~\ref{S:balance} offers a particular IM construction---building on my vacuous prior IM constructions in previous work---that incorporates partial prior information via a suitable belief function and Dempster's rule of combination.  Then Proposition~\ref{prop:dempster} shows that, under certain conditions on the prior belief function, this partial prior IM is valid.  A simple example illustrates this proposal and shows how this IM construction might strike the desired balance: there's clear improvement compared to the generalized Bayes IM when the prior information is very limited, and even improvements compared to the vacuous prior IM when the data is compatible with the prior information.  

More recently, I've advanced a very general partial prior IM construction \citep{martin.partial2, martin.partial3} that achieves an even stronger notion of validity, what I call {\em strong validity}.  As the name suggests, strong validity is stronger than validity, which is easy to see from its representation as a ``uniform version'' of \eqref{eq:valid}:
\[ \uprob_{Y,\Theta}\{ \lPi_Y(H) > 1-\alpha \text{ for some $H$ with $H \not\ni \Theta$}\} \leq \alpha, \quad \alpha \in [0,1]. \]
See \citet{cella.martin.probing} for some important implications of this uniformity.  It turns out that strong validity is closely tied to consonance, possibility theory, and the construction of what I refer to as a {\em possibilistic IM} relies on outer consonant approximations \citep[e.g.,][]{montes.etal.2019, dubois.prade.1990} and the so-called imprecise-probability-to-possibility transform \citep[e.g.,][]{dubois.etal.2004, hose2022thesis}.  Consequently, possibilistic IMs have structure that the generalized Bayes and Dempster's rule-based IMs don't, and I argue in \citet{martin.partial2} that despite being largely statistically motivated and satisfying a number of important statistical properties, it has some desirable behavioral properties as well.  To summarize, the Venn diagram in Figure~\ref{fig:venn} illustrates that both strongly valid and invulnerable IMs are special cases of valid IMs, and suggests an interesting open question: are there any possibilistic IMs that are also invulnerable?  If so, then that IM would arguably settle the statistical--behavioral reliability balance question.

\begin{figure}
\begin{center}
\scalebox{0.8}{
\begin{tikzpicture}
\begin{scope} [fill opacity = .3]
\draw[fill=lightgray, draw=lightgray, rounded corners] (-7,6) rectangle (7,-4); 
\draw[fill=cyan, draw=cyan, rounded corners] (-6,5) rectangle (6,-3);
\draw[fill=red, draw=red] (-2,1) circle (3);
\draw[fill=blue, draw=blue] (2,1) circle (3);
\end{scope}
\node at (-5.5,-3.5) {\bf no-sure-loss}; 
\node at (-5, -2.5) {\bf valid};
\node at (-2.2,3.1) {\bf strongly valid};
\node at (2.2,3.1) {\bf invulnerable};
\node at (3.3, 0) {GenBayes}; 
\draw[fill=black] (2.2, 0) circle (0.08); 
\node at (0.55, -2.5) {Dempster}; 
\draw[fill=black] (-0.5,-2.5) circle (0.08); 
\node at (-2.5, -0.5) {Possibilistic}; 
\draw[fill=black] (-3.7, -0.5) circle (0.08); 
\node at (0,1) {\large\bf {\color{yellow}???}}; 
\end{tikzpicture}
}
\end{center}
\caption{Venn diagram to illustration the results presented in this paper.  The specific IM constructions---generalized Bayes, Dempster's rule, and possibilistic---are highlighted by points in respective regions they fall into.}
\label{fig:venn}
\end{figure}
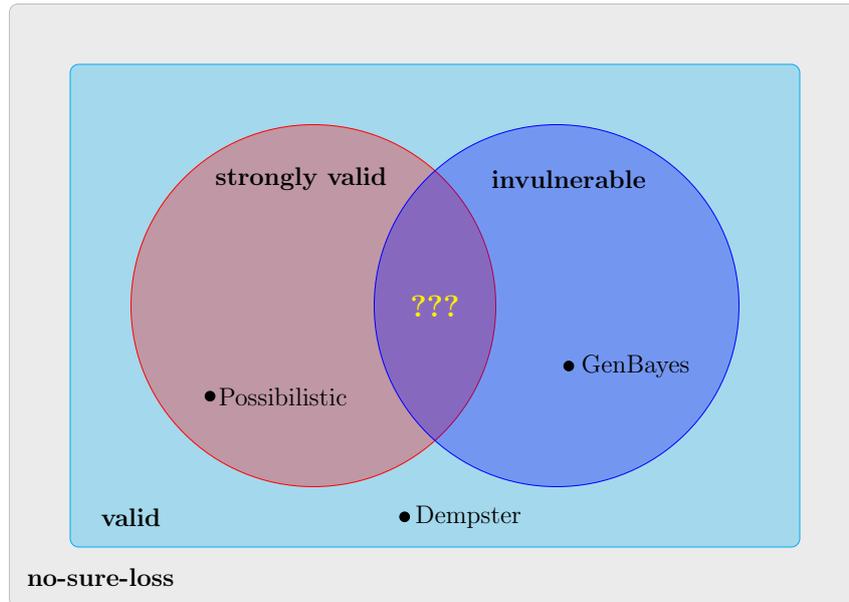

\section*{Acknowledgments}

Thanks to Teddy Seidenfeld for suggesting Fisher's ``underworld paper'' to me, and also for his critique following my {\em SIPTA Seminar} presentation (\url{https://www.youtube.com/watch?v=1VL5P1YcXdo}).  And special thanks to Enrique Miranda for some corrections and helpful suggestions concerning an earlier draft of this manuscript.  This work is partially supported by the U.S.~National Science Foundation, grant SES--2051225.






\bibliographystyle{apalike}
\bibliography{/Users/rgmarti3/Dropbox/Research/mybib}

\end{document}